\documentclass{amsart}
\usepackage[margin=1in]{geometry}
\usepackage{amsmath}
\usepackage{amssymb}
\usepackage{amsthm}
\usepackage{cite}
\usepackage[verbose]{hyperref}
\usepackage{booktabs}
\usepackage{enumitem}
\usepackage{array}

\theoremstyle{plain}
\newtheorem{theorem}{Theorem}[section]
\newtheorem{proposition}[theorem]{Proposition}
\newtheorem{lemma}[theorem]{Lemma}

\theoremstyle{definition}
\newtheorem{definition}[theorem]{Definition}

\theoremstyle{remark}
\newtheorem{remark}[theorem]{Remark}

\newcommand{\RR}{\mathbb{R}}
\newcommand{\CC}{\mathbb{C}}
\newcommand{\HH}{\mathbb{H}}
\newcommand{\FF}{\mathbb{F}}
\newcommand{\OO}{\mathbb{O}}
\newcommand{\Jac}{\mathrm{rad}}
\newcommand{\LL}{\ell\ell}

\title[Near-Trinions]{Near-Trinions: Complete Classification of\\
  Unital Three-Dimensional Real Associative Algebras}

\author{Joel A.\ Shelton}
\address{Department of Mathematics, Tusculum University,
60 Shiloh Road, Greeneville, TN 37745, USA}
\email{jshelton@tusculum.edu}

\subjclass[2020]{16D70, 16G30, 16P10, 16S70}

\keywords{real division algebras; Frobenius theorem;
Wedderburn--Artin theorem; Jacobson radical; classification of algebras;
near-trinions; nilpotent extensions; Peirce decomposition}

\date{Preprint, 2026}

\begin{document}

\begin{abstract}
Frobenius' theorem settles, definitively, the question of three-dimensional real division algebras: there are none. What it leaves open --- and what this paper addresses --- is the question of which three-dimensional real associative algebras actually exist. We call these \emph{near-trinions} and classify them completely, up to $\RR$-algebra isomorphism. The answer is exactly six isomorphism classes; we provide canonical representatives, explicit multiplication tables, and a collection of module-theoretic invariants sufficient to distinguish every pair. The classification proceeds by stratifying on the radical dimension $d$: Wedderburn--Artin resolves the semisimple stratum; in the $d=1$ stratum, the Peirce decomposition locates the radical generator in either a diagonal or off-diagonal corner, producing the commutative and non-commutative branches respectively and ruling out the residue field $\CC$ by a characteristic sign obstruction; and the $d=2$ stratum splits on whether $\Jac(A)^2$ vanishes. The argument not only recovers the six classes but explains why the radical dimension takes values in $\{0,1,2\}$, why the $d=1$ case splits on commutativity, and why no near-trinion with residue field $\CC$ can exist. One direction for subsequent work is proposed.
\end{abstract}

\maketitle

\section{Introduction}

\subsection{The problem and its origin}

The sequence of real normed division algebras --- the real numbers $\RR$,
the complex numbers $\CC$, the quaternions $\HH$, and the octonions $\OO$
--- occupies dimensions one, two, four, and eight, a pattern that admits
a structural explanation.
Each algebra in the sequence arises from its predecessor by the
\emph{Cayley--Dickson construction} \cite{Baez2002}, which doubles the dimension at every
step while sacrificing one algebraic property: $\RR$ is an ordered field;
$\CC$ loses ordering but retains commutativity; $\HH$ loses commutativity
but retains associativity; $\OO$ loses associativity but retains the
weaker alternative law.
What this construction cannot produce is an algebra of dimension three.
The doubling mechanism jumps from two to four, and the gap between them
is not a missing algebra in the sequence but a structural impossibility ---
one that Frobenius made precise.

\begin{theorem}[Frobenius, 1878 {\cite{Frobenius1878}}]\label{thm:frobenius}
  The only finite-dimensional associative real division algebras are
  $\RR$, $\CC$, and $\HH$.
\end{theorem}

A three-dimensional real division algebra --- an object one might
naturally call a \emph{trinion}, by analogy with the Latin sequence
\emph{singuli, bini, trini, quaterni} --- does not exist.
The proof is short: any such candidate must contain $\CC$ as a
subalgebra, after which any additional generator $j$ must satisfy
$j^2 = -\lambda^2$ for some $\lambda \in \RR^*$, making
$(j/\lambda)^2 = -1$ and producing a second copy of $\CC$ independent
of the first.
The three elements $1$, $i$, and $j/\lambda$ are then linearly
independent over $\RR$ and their product $ij/\lambda$ contributes a
fourth, driving the dimension to at least four and leaving no room for
a trinion.

A self-contained presentation of this argument appears in \cite{Shelton2024};
the present paper is, in a sense, the natural sequel.

The question is simply this: Frobenius' theorem tells us that no
three-dimensional \emph{division} algebra over $\RR$ exists, but it says
nothing about three-dimensional associative algebras in general.
The class of all such algebras --- the \emph{near-trinions} --- has exactly
six members, as Theorem~\ref{thm:main} establishes.

\subsection{What is and is not in the literature}

The classification problem for low-dimensional associative algebras has
since accumulated a substantial literature.
The three-dimensional case over $\RR$ was first addressed by Scheffers
\cite{Scheffers1891} in 1891.
The modern treatment --- complete, rigorous, and explicitly over both $\RR$
and $\CC$ --- appears in Kobayashi, Shirayanagi, Takahasi, and Tsukada
\cite{Kobayashi2021}, whose six isomorphism classes agree with
Theorem~\ref{thm:main}; the correspondence is made explicit in
Proposition~\ref{prop:kst}.

Given this prior work, the contributions of the present paper are
the following.

\emph{Framework.}
The near-trinion framing gives the classification a narrative coherence
and a connection to the long history of the problem that the earlier
treatments, which approach the question as a routine enumeration, do not provide.

\emph{Proof method.}
The argument in \cite{Kobayashi2021} is coordinate-based: structure
constants are varied until all isomorphism classes are enumerated.
It produces the correct list but does not explain \emph{why} the list
has the structure it does.
The present argument proceeds through the Jacobson radical stratification
and Peirce decomposition (Sections~\ref{sec:prelim} and
\ref{sec:classification}).
This approach answers a different question: not only \emph{which} algebras
exist, but \emph{why} the radical dimension $d$ takes values in $\{0,1,2\}$,
why the $d=1$ case splits on commutativity, and why no near-trinion with
residue field $\CC$ can exist.
The last point merits proof here, since it drives the entire $d=1$ analysis. If $A/ \Jac(A)\cong\CC$ and $\Jac(A)=\RR\eta$ for some nonzero $\eta\in\Jac(A)$, then $i\cdot\eta=c\eta$ for some $c\in\RR$,
so applying $i$ again gives $i^2\cdot\eta = c^2\eta$. Since $i^2=-1 \in \CC$, we obtain
$(-1)\eta = c^2\eta$, hence $c^2=-1$, impossible over $\RR$.
The $d=1$ stratum therefore splits into exactly two near-trinions, determined
entirely by whether the radical generator occupies a diagonal or off-diagonal
Peirce corner subalgebra. 

\subsection{Road map}

The proof of Theorem~\ref{thm:main} is organized as follows.
Section~\ref{sec:prelim} establishes the three tools used throughout:
the Jacobson radical (three equivalent characterizations and their
computations for each near-trinion), the Wedderburn--Artin theorem,
and Nakayama's Lemma.
Section~\ref{sec:classification} carries out the classification in three
strata determined by $d = \dim_\RR \Jac(A) \in \{0,1,2\}$.
In the $d = 0$ stratum, Wedderburn--Artin identifies $\RR^3$ and
$\RR\times\CC$ as the only semisimple candidates.
In the $d = 1$ stratum, the Peirce decomposition shows that the radical
generator lies either in a diagonal corner (producing the commutative
near-trinion NT3) or an off-diagonal corner (producing the non-commutative
near-trinion NT6); the subcase $A/\Jac(A) \cong \CC$ is shown to be
impossible.
In the $d = 2$ stratum, whether $\Jac(A)^2$ vanishes distinguishes
NT4 from NT5.
Section~\ref{sec:invariants} establishes irredundancy: the six algebras
are distinguished by $d$, nilpotency of $\Jac(A)^2$, commutativity,
and primitive idempotent count; specifically, $d$ alone separates the
three strata; commutativity separates NT3 from NT6; and Loewy length
separates NT4 from NT5. One direction for subsequent work is described in
Section~\ref{sec:future}.

\subsection{Conventions}

Throughout, \emph{algebra} means a finite-dimensional, associative,
unital $\RR$-algebra, unless stated otherwise.
The Jacobson radical of an algebra $A$ is denoted $\Jac(A)$, and we write
$\LL(A)$ for the Loewy length, the least $n$ such that $\Jac(A)^n = 0$.
The phrase \emph{near-trinion} refers specifically to
any three-dimensional associative unital $\RR$-algebra.
For a quotient ring $R/I$, an overbar denotes the coset of an element:
$\bar{r} = r + I \in R/I$; thus $\bar\varepsilon \in \RR[\varepsilon]/(\varepsilon^2)$,
$\bar x \in \RR[x]/(x^3)$, and so on.
Peirce corners are written $e_i A e_k = \{e_i a e_k : a \in A\}$
for idempotents $e_i, e_k$. 

\bigskip

\begin{theorem}[Near-Trinion Classification]\label{thm:main}
  Up to $\RR$-algebra isomorphism, every near-trinion is isomorphic to
  exactly one of:
  \begin{enumerate}[label=\textup{(NT\arabic*)}]
    \item $\RR^3 = \RR \times \RR \times \RR$
    \item $\RR \times \CC$
    \item $\RR \times \RR[\varepsilon]/(\varepsilon^2)$
    \item $\RR[\varepsilon_1, \varepsilon_2]/
          (\varepsilon_1^2,\, \varepsilon_2^2,\, \varepsilon_1\varepsilon_2)$
    \item $\RR[x]/(x^3)$
    \item $T_2(\RR)$, the algebra of upper triangular $2\times 2$ real matrices
  \end{enumerate}
  This list is irredundant: no two of the above are isomorphic.
  The six algebras are distinguished by the invariants in
  Section~\ref{sec:invariants}: radical dimension $d$ separates the
  three strata $\{$NT1,\,NT2$\}$, $\{$NT3,\,NT6$\}$, $\{$NT4,\,NT5$\}$;
  commutativity separates NT3 from NT6; and Loewy length separates
  NT4 from NT5.
\end{theorem}

The proof occupies Sections~\ref{sec:prelim} and~\ref{sec:classification};
the invariants that establish irredundancy are computed in
Section~\ref{sec:invariants}.

\begin{center}
\renewcommand{\arraystretch}{1.25}
\begin{tabular}{clll}
  \toprule
  Stratum & Algebras & Separating invariant(s) \\
  \midrule
  $d=0$ & NT1 vs.\ NT2 & idempotent count (3 vs.\ 2) \\
  $d=1$ & NT3 vs.\ NT6 & commutativity (Yes vs.\ No) \\
  $d=2$ & NT4 vs.\ NT5 & $\Jac(A)^2=0$? (Yes vs.\ No); $\LL(A)$ (2 vs.\ 3) \\
  \bottomrule
\end{tabular}
\end{center}

\section{Preliminaries}\label{sec:prelim}

We record here the ring-theoretic tools on which the classification rests.
The reader familiar with the Wedderburn--Artin theorem and the Jacobson
radical may proceed directly to Section~\ref{sec:classification}.

\subsection{The Jacobson radical: three characterizations}

For finite-dimensional algebras over a field, the Jacobson radical admits
three descriptions that are, in general, distinct objects; the content of
the following proposition is that they coincide.

\begin{definition}\label{def:jac}
  Let $A$ be a finite-dimensional $\RR$-algebra.
  Define the following subsets of $A$:
  \begin{enumerate}[label=\textup{(\alph*)}]
    \item $J_1(A) = \bigcap \{\mathfrak{m} : \mathfrak{m}
          \text{ a maximal left ideal of } A\}$;
    \item $J_2(A) = $ the sum of all nilpotent left ideals of $A$;
    \item $J_3(A) = \{a \in A : 1 - ba \text{ is invertible for all }
          b \in A\}$.
  \end{enumerate}
\end{definition}

\begin{proposition}[Equivalence of characterizations {\cite{Pierce1982}}]
  \label{prop:jaceq}
  For any finite-dimensional $\RR$-algebra $A$,
  $J_1(A) = J_2(A) = J_3(A)$.
  This common object is the \textbf{Jacobson radical} $\Jac(A)$,
  and it is the unique largest nilpotent two-sided ideal of $A$.
\end{proposition}

\begin{remark}
  The three characterizations illuminate different aspects of the radical.
  $J_1$ is computed directly in Section~\ref{sec:maxideals}.
  $J_2$ is used in the classification: the radical is nilpotent and
  bounded in degree.
  $J_3$ provides a practical test: $a \in \Jac(A)$ if and only if
  $1 - a$ is invertible.
\end{remark}

Recall that the Loewy length $\LL(A)$ is defined in Section~1.4
(Conventions); for $\RR[x]/(x^3)$ it takes the value $3$, the maximum
possible in dimension three.

\subsection{The Wedderburn--Artin theorem}

\begin{theorem}[Wedderburn--Artin {\cite{Wedderburn1908,Pierce1982}}]
  \label{thm:wa}
  Every semisimple finite-dimensional algebra over a field $k$ is isomorphic
  to a finite direct product of full matrix rings:
  \[
    A \;\cong\; M_{n_1}(D_1) \times \cdots \times M_{n_r}(D_r),
  \]
  where each $D_i$ is a finite-dimensional division algebra over $k$.
  When $k = \RR$, Frobenius' theorem restricts each $D_i$ to
  $\RR$, $\CC$, or $\HH$.
\end{theorem}

\subsection{Comparison with Kobayashi--Shirayanagi--Takahasi--Tsukada}
\label{sec:kst}

\begin{proposition}[Correspondence with {\cite{Kobayashi2021}}]
  \label{prop:kst}
  The six near-trinions correspond to the unital entries in
  \cite{Kobayashi2021} as given in Table~\ref{tab:kst}; each
  near-trinion is matched to a unique entry by its invariant signature.
\end{proposition}

\begin{table}[h]
\caption{The six near-trinions and their corresponding entries in
\cite{Kobayashi2021}. Here $d = \dim_\RR \Jac(A)$ and $J^2 = \Jac(A)^2$;
the final column gives table numbers in Section~4 of \cite{Kobayashi2021}.}
\label{tab:kst}
\renewcommand{\arraystretch}{1.35}
\begin{tabular}{llll}
  \toprule
  Near-trinion & Invariant signature & Label in \cite{Kobayashi2021} & Table (in \cite{Kobayashi2021}) \\
  \midrule
  \textup{(NT1)} $\RR^3$
    & $d=0$;\ 3 idempotents
    & $U^3_0$ & (17) \\
  \textup{(NT2)} $\RR \times \CC$
    & $d=0$;\ 2 idempotents;\ $\CC$-factor
    & $U^3_{2-}{}^{\dagger}$ & (18) \\
  \textup{(NT3)} $\RR \times \RR[\varepsilon]/(\varepsilon^2)$
    & $d=1$;\ $J^2=0$;\ commutative
    & $U^3_1$ & (19) \\
  \textup{(NT4)} $\RR[\varepsilon_1,\varepsilon_2]/(\varepsilon^2_1,\varepsilon^2_2,\varepsilon_1\varepsilon_2)$
    & $d=2$;\ $J^2=0$;\ 1 idempotent
    & $U^3_3$ & (25) \\
  \textup{(NT5)} $\RR[x]/(x^3)$
    & $d=2$;\ $J^2\neq 0$;\ 1 idempotent
    & $U^3_2$ & (20) \\
  \textup{(NT6)} $T_2(\RR)$
    & $d=1$;\ $J^2=0$;\ non-commutative
    & $U^3_4$ & (26) \\
  \bottomrule
\end{tabular}
\smallskip

\noindent\footnotesize{$^\dagger\,U^3_{2-}$ (subscript with minus sign)
denotes the algebra existing only over $\RR$, arising from the real division
algebra $\CC$; it has no analogue over $\CC$ and accounts for the count of
six (over $\RR$) versus five (over $\CC$) unital three-dimensional algebras.
The label $U^3_2$ (no minus sign, row NT5) is a distinct algebra.}
\end{table}

\begin{proof}
  Each near-trinion is uniquely identified by its invariant signature.
  The matching against \cite{Kobayashi2021} is read off directly:
  $U^3_0$ carries three orthogonal idempotents; $U^3_{2-}$ (the minus-subscript
  label, denoting the $\RR$-only algebra) has two idempotents and a $\CC$-factor;
  $U^3_1$ is the commutative split extension with $d=1$;
  $U^3_2$ (no minus, distinct from $U^3_{2-}$) has Loewy length three;
  $U^3_3$ has square-zero two-dimensional radical; and $U^3_4$ is the unique
  non-commutative entry.
\end{proof}

\begin{remark}
  The label $U^3_{2-}$ in \cite{Kobayashi2021} denotes an algebra
  specific to $\RR$: it arises from the real division algebra $\CC$,
  which has no complex analogue.
  This is why \cite{Kobayashi2021} lists five unital algebras over $\CC$
  but six over $\RR$.
\end{remark}

\subsection{Maximal ideals of near-trinions}
\label{sec:maxideals}

\begin{proposition}\label{prop:maxideals}
  The maximal left ideals of the six near-trinions, with $\Jac(A)$
  confirmed in each case as their intersection, are as follows.
\end{proposition}

\begin{proof}
We work through each class in turn.

\smallskip\noindent
\textbf{(NT1)} $A = \RR^3$, basis $\{e_1, e_2, e_3\}$, $e_ie_j = \delta_{ij}e_i$.
Since $\RR^3$ is semisimple and each $e_i$ generates a minimal ideal
$\RR e_i \cong \RR$, the maximal left ideals are the codimension-one
complementary direct sums:
\[
  \mathfrak{m}_1 = \RR e_2 \oplus \RR e_3,\quad
  \mathfrak{m}_2 = \RR e_1 \oplus \RR e_3,\quad
  \mathfrak{m}_3 = \RR e_1 \oplus \RR e_2.
\]
Their intersection is $0 = \Jac(\RR^3)$.

\smallskip\noindent
\textbf{(NT2)} $A = \RR \times \CC$.
The two minimal direct summands are $\RR \times 0$ and $0 \times \CC$.
The maximal left ideals are their complements:
$\mathfrak{m}_1 = 0 \times \CC$ and $\mathfrak{m}_2 = \RR \times 0$.
Their intersection is $0 = \Jac(\RR \times \CC)$.

\smallskip\noindent
\textbf{(NT3)} $A = \RR \times \RR[\varepsilon]/(\varepsilon^2)$.
The first factor contributes the maximal ideal
$\mathfrak{m}_1 = 0 \times \RR[\varepsilon]/(\varepsilon^2)$.
The second factor $\RR[\varepsilon]/(\varepsilon^2)$ is local with
unique maximal ideal $(\bar\varepsilon)$, contributing
$\mathfrak{m}_2 = \RR \times (\bar\varepsilon)$.
Their intersection is
$\mathfrak{m}_1 \cap \mathfrak{m}_2
 = \{(0, a\bar\varepsilon) : a \in \RR\}
 = 0 \times (\bar\varepsilon) = \Jac(A)$,
one-dimensional, consistent with $d = 1$.

\smallskip\noindent
\textbf{(NT4)} $A = \RR[\varepsilon_1, \varepsilon_2]/(\varepsilon_1^2, \varepsilon_2^2, \varepsilon_1 \varepsilon_2)$.
This algebra is local: an element $a + b\varepsilon_1 + c\varepsilon_2$ is invertible
if and only if $a \neq 0$, with inverse
$a^{-1} - a^{-2}b\varepsilon_1 - a^{-2}c\varepsilon_2$ (using $\varepsilon_a^2 = \varepsilon_1 \varepsilon_2 = 0$).
The unique maximal left ideal is
$\mathfrak{m} = \RR \varepsilon_1 \oplus \RR \varepsilon_2 = \Jac(A)$,
two-dimensional, consistent with $d = 2$.

\smallskip\noindent
\textbf{(NT5)} $A = \RR[x]/(x^3)$.
An element $a_0 + a_1\bar x + a_2\bar x^2$ is invertible iff $a_0 \neq 0$.
The unique maximal ideal is
$\mathfrak{m} = \RR\bar x \oplus \RR\bar x^2 = \Jac(A)$,
two-dimensional with $\Jac(A)^2 = \RR\bar x^2 \neq 0$, so $\LL(A) = 3$.

\smallskip\noindent
\textbf{(NT6)} $A = T_2(\RR)$, basis $\{e_{11}, e_{22}, e_{12}\}$.
An element $ae_{11} + be_{22} + ce_{12}$ is invertible iff $a \neq 0$
and $b \neq 0$.
The two maximal left ideals are:
\[
  \mathfrak{m}_1 = \RR e_{22} \oplus \RR e_{12}
  \quad\text{and}\quad
  \mathfrak{m}_2 = \RR e_{11} \oplus \RR e_{12}.
\]
Their intersection is
$\mathfrak{m}_1 \cap \mathfrak{m}_2 = \RR e_{12} = \Jac(A)$,
one-dimensional, consistent with $d = 1$.
\end{proof}

\begin{table}[h]
\caption{Maximal left ideals of the six near-trinions. In each case the
intersection $\bigcap_i \mathfrak{m}_i = \Jac(A)$, confirming
Proposition~\ref{prop:jaceq}. Overbars denote cosets as in Conventions;
$\varepsilon_a^2$ abbreviates $\varepsilon_1^2$ and $\varepsilon_2^2$
(index $a \in \{1,2\}$).}
\label{tab:maxideals}
\renewcommand{\arraystretch}{1.35}
\begin{tabular}{lccll}
  \toprule
  Algebra & $d$ & Number of maximal left ideals &
  Intersection $\bigcap_i \mathfrak{m}_i$ & Equals $\Jac(A)$\\
  \midrule
  \textup{(NT1)} $\RR^3$   & 0 & 3 & $0$                              & Yes \\
  \textup{(NT2)} $\RR\times\CC$  & 0 & 2 & $0$                        & Yes \\
  \textup{(NT3)} $\RR \times \RR[\varepsilon]/(\varepsilon^2)$
                           & 1 & 2 & $\{0\}\times\RR\bar\varepsilon$   & Yes \\
  \textup{(NT4)} $\RR[\varepsilon_1,\varepsilon_2]/(\varepsilon_a^2,\varepsilon_1\varepsilon_2)$
                           & 2 & 1 & $\RR \varepsilon_1\oplus\RR \varepsilon_2$            & Yes \\
  \textup{(NT5)} $\RR[x]/(x^3)$
                           & 2 & 1 & $\RR\bar x\oplus\RR\bar x^2$      & Yes \\
  \textup{(NT6)} $T_2(\RR)$ & 1 & 2 & $\RR e_{12}$                    & Yes \\
  \bottomrule
\end{tabular}
\end{table}

\begin{remark}
  The local near-trinions --- (NT4) and (NT5) --- are the only ones with
  a single maximal left ideal; all others have two or three.
  Among the non-local algebras with $d \geq 1$, both (NT3) and (NT6)
  have two maximal left ideals, but (NT3) is commutative and (NT6) is not.
\end{remark}

\subsection{The Nakayama Lemma for near-trinions}
\label{sec:nakayama}

\begin{theorem}[Nakayama's Lemma {\cite{Pierce1982}}]
  \label{thm:nakayama}
  Let $A$ be a finite-dimensional algebra and $M$ a finitely generated
  left $A$-module.
  If $\Jac(A) \cdot M = M$, then $M = 0$.
\end{theorem}

For a near-trinion $A$, set $d = \dim_\RR \Jac(A)$.
Since $\dim_\RR A = 3$ we have $d \in \{0,1,2,3\}$ a priori, but $d = 3$
is ruled out: if $\Jac(A)$ were three-dimensional it would equal $A$,
making $A$ nilpotent.
But $1 \in A$ is not nilpotent, a contradiction, so $1 \notin \Jac(A)$, forcing $d \leq 2$.
The semisimple quotient $A/\Jac(A)$ has dimension $3 - d$, and the three
strata are:

\begin{itemize}
  \item $d = 0$: $A$ semisimple; Wedderburn--Artin applies directly,
    yielding (NT1) and (NT2).

  \item $d = 1$: $A/\Jac(A)$ has dimension two.
    Two-dimensional semisimple real algebras are either $\RR^2$ or $\CC$.
    The subcase $A/\Jac(A) \cong \CC$ is impossible in dimension three
    (shown in Section~\ref{sec:d1}).
    The subcase $A/\Jac(A) \cong \RR^2$ yields two distinct near-trinions,
    (NT3) and (NT6), distinguished by whether the radical generator
    lies in a diagonal or off-diagonal Peirce corner.

  \item $d = 2$: $A/\Jac(A) \cong \RR$; $A$ is local with two-dimensional
    nil radical; whether $\Jac(A)^2 = 0$ distinguishes (NT4) from (NT5).
\end{itemize}

\begin{center}
\renewcommand{\arraystretch}{1.3}
\begin{tabular}{ccc}
  \toprule
  $d = \dim_\RR \Jac(A)$ & $\dim_\RR A/\Jac(A)$ & Candidates \\
  \midrule
  $0$ & $3$ & (NT1), (NT2) \\
  $1$ & $2$ & (NT3), (NT6) \\
  $2$ & $1$ & (NT4), (NT5) \\
  \bottomrule
\end{tabular}
\end{center}

\section{The Classification}\label{sec:classification}

The three strata $d \in \{0,1,2\}$ partition the problem exhaustively
(Section~\ref{sec:prelim}), and within each stratum the analysis below is
complete: the $d=0$ stratum is settled by a single dimension count;
the $d=1$ stratum has two Peirce-corner sub-cases (diagonal and off-diagonal)
plus one impossibility argument; the $d=2$ stratum splits on whether
$\Jac(A)^2=0$.

\subsection{The semisimple case (\texorpdfstring{$d=0$}{d=0})}

When $\Jac(A) = 0$, Theorem~\ref{thm:wa} applies directly.
The dimension constraint reads: $\dim_\RR M_n(D) = n^2 \dim_\RR D$, so
the only ways to write $3$ as such a sum are $3 = 1+1+1$ and $3 = 1+2$.
The first gives $\RR \times \RR \times \RR$; the second gives $\RR \times \CC$.
No other combination works: $\dim \HH = 4$ and $\dim M_2(\RR) = 4$
already exceed three.

\begin{proposition}\label{prop:ss}
  The semisimple near-trinions are, up to isomorphism,
  \textup{(NT1)} $\RR^3$ and \textup{(NT2)} $\RR \times \CC$.
\end{proposition}

\subsection{One-dimensional radical (\texorpdfstring{$d=1$}{d=1})}
\label{sec:d1}
Here $A/\Jac(A)$ is a two-dimensional semisimple $\RR$-algebra, 
hence isomorphic to either $\RR^2$ or $\CC$. The subcase 
$A/\Jac(A) \cong \RR^2$ produces two non-isomorphic near-trinions, 
determined by whether the radical generator occupies a diagonal or 
off-diagonal Peirce corner; the subcase $A/\Jac(A) \cong \CC$ is impossible over $\RR$,
as the argument below establishes.

\medskip
\noindent\textbf{Subcase:} $A/\Jac(A) \cong \RR^2$.
Choose a basis $\{e_1, e_2, j\}$ for $A$ where $e_1, e_2$ are
orthogonal idempotents summing to $1$ and $j$ generates $\Jac(A)$.
Nilpotency forces $j^2 = 0$.
(Here $j$ is a name for the radical generator in the $d=1$ analysis;
it is unrelated to the basis elements $\varepsilon_1, \varepsilon_2$ used for NT4 in
Section~\ref{sec:classification}.)
The structure is determined by how $j$ sits in the Peirce decomposition
(the decomposition of $A$ into corner subalgebras $e_i A e_k$ determined
by the idempotents $e_1, e_2$; see Section~\ref{sec:prelim} and
Conventions)
with respect to $e_1$ and $e_2$: since $j$ is a non-unit, $j$ belongs
to some corner $e_i A e_k$, meaning $e_i \cdot j = j$ and $j \cdot e_k = j$.
There are two essentially distinct cases.

\medskip
\noindent\textbf{Case (a): diagonal placement} ($i = k$).
Without loss of generality $j \in e_1 A e_1$, so $e_1 j = j e_1 = j$
and $e_2 j = j e_2 = 0$.
Here $j$ commutes with both idempotents, and the algebra decomposes as
$A \cong (e_1 A e_1) \times (e_2 A e_2)$.
Since $e_1 A e_1 \cong \RR[\varepsilon]/(\varepsilon^2)$ and
$e_2 A e_2 \cong \RR$, we obtain $A \cong \RR \times \RR[\varepsilon]/(\varepsilon^2)$.
This algebra is commutative.
\[
  \boxed{A \;\cong\; \RR \times \RR[\varepsilon]/(\varepsilon^2). \tag{NT3}}
\]
Hence, Case~(a) yields exactly \textup{(NT3)}.

\medskip
\noindent\textbf{Case (b): off-diagonal placement} ($i \neq k$).
Without loss of generality $j \in e_1 A e_2$, so
$e_1 j = j$, $j e_2 = j$, while $e_2 j = 0$ and $j e_1 = 0$.
The multiplication table on $\{e_1, e_2, j\}$ is:
\[
  e_1^2 = e_1,\quad e_2^2 = e_2,\quad e_1 e_2 = 0,\quad
  e_1 j = j,\quad j e_2 = j,\quad
  e_2 j = j e_1 = j^2 = 0.
\]
Under the identification $e_1 = e_{11}$, $e_2 = e_{22}$, $j = e_{12}$,
this is exactly $T_2(\RR)$.
The algebra is \emph{non-commutative}: $e_1 j = j \neq 0 = j e_1$.
\[
  \boxed{A \;\cong\; T_2(\RR). \tag{NT6}}
\]
Hence, Case~(b) yields exactly \textup{(NT6)}.

The two cases produce non-isomorphic algebras: (NT3) is commutative and
(NT6) is not.
Thus the subcase $A/\Jac(A) \cong \RR^2$ yields exactly the two near-trinions
(NT3) and (NT6).

\medskip
\noindent\textbf{Subcase:} $A/\Jac(A) \cong \CC$: \textbf{impossible in dimension three.}

If $A/\Jac(A) \cong \CC$ and $\Jac(A) = \RR\eta$ (one-dimensional), then
$\Jac(A)$ is a two-sided ideal, so $i \cdot \eta \in \Jac(A) = \RR\eta$,
giving $i \cdot \eta = c\eta$ for some $c \in \RR$.
But then $i^2 \cdot \eta = c^2\eta$, and since $i^2 = -1$ in $\CC$ we
need $c^2 = -1$, which has no real solution.
Hence, the subcase $A/\Jac(A) \cong \CC$ is impossible, and the $d = 1$
case yields exactly (NT3) and (NT6).

\subsection{Two-dimensional radical (\texorpdfstring{$d=2$}{d=2})}

Here $A/\Jac(A) \cong \RR$ and $A$ is local with two-dimensional nil
radical $J = \Jac(A)$.
The product $J^2$ is either zero or one-dimensional.

\medskip
\noindent\textbf{Subcase:} $J^2 = 0$.
Taking a basis $\{1, \varepsilon_1, \varepsilon_2\}$ with $\varepsilon_a \varepsilon_b = 0$ for all $a, b$:
\[
  \boxed{A \;\cong\;
  \RR[\varepsilon_1, \varepsilon_2]\big/(\varepsilon_1^2,\, \varepsilon_2^2,\, \varepsilon_1 \varepsilon_2). \tag{NT4}}
\]
The isomorphism class is independent of the choice of basis for $J$:
any two bases are related by an invertible linear map $\phi\colon J\to J$.
Since $J^2=0$, every product of two elements of $J$ is zero, so any invertible
linear change of basis $J\to J$ fixes all products and hence extends (with
$1\mapsto 1$) to an $\RR$-algebra automorphism of $A$.

Hence, the subcase $J^2=0$ yields exactly \textup{(NT4)}.

\medskip
\noindent\textbf{Subcase:} $J^2 \neq 0$, so $\dim_\RR J^2 = 1$.
Let $x \in J \setminus J^2$.
Since $x \notin J^2 = \RR x^2$ and $\dim_\RR J = 2$, the elements
$x$ and $x^2$ are linearly independent, so $\{x, x^2\}$ is a basis for $J$.
Since $J^3 \subseteq J \cdot J^2 = \RR x \cdot \RR x^2 = \RR x^3$,
and nilpotency of $J$ forces $J^N = 0$ for some $N$, we conclude
$J^3 = 0$ and hence $x^3 = 0$.
Thus
\[
  \boxed{A \;\cong\; \RR[x]/(x^3). \tag{NT5}}
\]
Any other generator $x' = ax + bx^2$ with $a \neq 0$ gives the same
isomorphism class. Indeed, the assignment $x \mapsto x'$ extends uniquely to an $\RR$-algebra
automorphism of $\RR[x]/(x^3)$, and $(x')^3 = 0$ since $x^3 = 0$. 
Hence, the subcase $J^2\neq 0$ yields exactly \textup{(NT5)}.

\subsection{Multiplication tables}
\label{sec:multtables}

We now record, for each near-trinion, a concrete multiplication table
on a chosen $\RR$-basis $\{v_1, v_2, v_3\}$.  Off-table products are
determined by $\RR$-bilinearity and the unit; entries $v_i \cdot v_j$
are read from row $i$, column $j$.  

\medskip\noindent
\textbf{(NT1)} $A = \RR^3$, basis $\{e_1, e_2, e_3\}$:
$e_i e_j = \delta_{ij} e_i$.

\medskip\noindent
\textbf{(NT2)} $A = \RR \times \CC$,
basis $\{f,\, \mathbf{1}_\CC,\, \mathbf{i}\}$
where $f = (1,0)$, $\mathbf{1}_\CC = (0,1)$, $\mathbf{i} = (0,i)$:

\begin{center}
\renewcommand{\arraystretch}{1.2}
\begin{tabular}{c|ccc}
  $\cdot$ & $f$ & $\mathbf{1}_\CC$ & $\mathbf{i}$ \\\hline
  $f$             & $f$              & $0$               & $0$ \\
  $\mathbf{1}_\CC$& $0$              & $\mathbf{1}_\CC$  & $\mathbf{i}$ \\
  $\mathbf{i}$    & $0$              & $\mathbf{i}$      & $-\mathbf{1}_\CC$ \\
\end{tabular}
\end{center}

\medskip\noindent
\textbf{(NT3)} $A = \RR \times \RR[\varepsilon]/(\varepsilon^2)$,
basis $\{e_1, e_2, j\}$ with $e_1 = (1,0)$, $e_2 = (0,1)$,
$j = (0,\bar\varepsilon)$:

\begin{center}
\renewcommand{\arraystretch}{1.2}
\begin{tabular}{c|ccc}
  $\cdot$ & $e_1$ & $e_2$ & $j$ \\\hline
  $e_1$ & $e_1$ & $0$   & $0$ \\
  $e_2$ & $0$   & $e_2$ & $j$ \\
  $j$   & $0$   & $j$   & $0$ \\
\end{tabular}
\end{center}
Here $j = e_2 j = j e_2$ and $j^2 = 0$;
$j$ lies in the diagonal corner $e_2 A e_2$, confirming Case (a) above.

\medskip\noindent
\textbf{(NT4)} $A = \RR[\varepsilon_1,\varepsilon_2]/(\varepsilon_1^2, \varepsilon_2^2, \varepsilon_1 \varepsilon_2)$,
basis $\{1, \varepsilon_1, \varepsilon_2\}$:

\begin{center}
\renewcommand{\arraystretch}{1.2}
\begin{tabular}{c|ccc}
  $\cdot$ & $1$ & $\varepsilon_1$ & $\varepsilon_2$ \\\hline
  $1$          & $1$          & $\varepsilon_1$ & $\varepsilon_2$ \\
  $\varepsilon_1$ & $\varepsilon_1$ & $0$ & $0$ \\
  $\varepsilon_2$ & $\varepsilon_2$ & $0$ & $0$ \\
\end{tabular}
\end{center}

\medskip\noindent
\textbf{(NT5)} $A = \RR[x]/(x^3)$, basis $\{1, \bar x, \bar x^2\}$:

\begin{center}
\renewcommand{\arraystretch}{1.2}
\begin{tabular}{c|ccc}
  $\cdot$ & $1$ & $\bar x$ & $\bar x^2$ \\\hline
  $1$         & $1$        & $\bar x$   & $\bar x^2$ \\
  $\bar x$    & $\bar x$   & $\bar x^2$ & $0$ \\
  $\bar x^2$  & $\bar x^2$ & $0$        & $0$ \\
\end{tabular}
\end{center}
The radical satisfies $\Jac(A)^2 = \RR\bar x^2 \neq 0$,
distinguishing (NT5) from (NT4).

\medskip\noindent
\textbf{(NT6)} $A = T_2(\RR)$,
basis $\{e_{11}, e_{22}, e_{12}\}$ with $e_{11} + e_{22} = 1_A$:

\begin{center}
\renewcommand{\arraystretch}{1.2}
\begin{tabular}{c|ccc}
  $\cdot$ & $e_{11}$ & $e_{22}$ & $e_{12}$ \\\hline
  $e_{11}$ & $e_{11}$ & $0$      & $e_{12}$ \\
  $e_{22}$ & $0$      & $e_{22}$ & $0$ \\
  $e_{12}$ & $0$      & $e_{12}$ & $0$ \\
\end{tabular}
\end{center}
Note $e_{11} \cdot e_{12} = e_{12}$ while $e_{12} \cdot e_{11} = 0$:
this is the unique non-commutative near-trinion.
The radical generator $e_{12}$ lies in the off-diagonal corner
$e_{11}A e_{22}$, confirming Case (b) above.
The (NT6) table is the unique one not symmetric under transposition.

\begin{lemma}\label{lem:complete}
  Every near-trinion falls into exactly one of the cases analysed above,
  and is therefore isomorphic to exactly one of \textup{(NT1)}--\textup{(NT6)}.
\end{lemma}

\begin{proof}
  The radical stratification partitions the problem into $d \in \{0,1,2\}$.
  Within each stratum the analysis is exhaustive: the semisimple case is
  settled by Proposition~\ref{prop:ss}; the $d = 1$ case branches on
  Peirce corner placement, yielding (NT3) and (NT6) (the subcase
  $A/\Jac(A) \cong \CC$ being impossible as shown); the $d = 2$ case
  branches on whether $J^2$ vanishes, yielding (NT4) and (NT5).
  No other branches exist.
  Irredundancy is the content of Section~\ref{sec:invariants}.
\end{proof}

\section{Invariants and Irredundancy}\label{sec:invariants}

The six near-trinions are distinguished by four invariants:
$d = \dim_\RR \Jac(A)$; whether $\Jac(A)^2 = 0$; commutativity; and the
number of primitive orthogonal idempotents.
Table~\ref{tab:invariants} below records these together with center
dimension and Loewy length.
For the five commutative near-trinions, $Z(A) = A$ so $\dim Z(A) = 3$;
for $T_2(\RR)$ the center consists only of scalar matrices,
giving $\dim Z(A) = 1$.

\begin{table}[h]
\caption{Distinguishing invariants of the six near-trinions.
Abbreviations: $d = \dim_\RR \Jac(A)$; N/A\,=\,not applicable
(radical already zero); Commut.\,=\,Commutative?;
$\LL(A)$\,=\,Loewy length.
Rows ordered by $d$; see Remark~\ref{rem:separation}.}
\label{tab:invariants}
\renewcommand{\arraystretch}{1.35}
\begin{tabular}{lcccccc}
  \toprule
  Algebra & $d$ & $\Jac(A)^2{=}0$? & Commut.? & Idempotents & $\dim Z$ & $\LL(A)$  \\
  \midrule
  \textup{(NT1)} $\RR^3$
    & 0 & N/A & Yes & 3 & 3 & 1 \\
  \textup{(NT2)} $\RR \times \CC$
    & 0 & N/A & Yes & 2 & 3 & 1 \\
  \textup{(NT3)} $\RR \times \RR[\varepsilon]/(\varepsilon^2)$
    & 1 & Yes & Yes & 2 & 3 & 2 \\
  \textup{(NT6)} $T_2(\RR)$
    & 1 & Yes & No  & 2 & 1 & 2 \\
  \textup{(NT4)} $\RR[\varepsilon_1,\varepsilon_2]/(\varepsilon_a^2, \varepsilon_1 \varepsilon_2)$
    & 2 & Yes & Yes & 1 & 3 & 2 \\
  \textup{(NT5)} $\RR[x]/(x^3)$
    & 2 & No  & Yes & 1 & 3 & 3 \\
  \bottomrule
\end{tabular}
\end{table}

\begin{remark}[Separation of every pair]\label{rem:separation}
The six near-trinions are separated by the following module-theoretic invariants of Table~\ref{tab:invariants}. The invariant $d$ alone separates into three strata; within $d=0$, the idempotent count distinguishes NT1 (three primitive idempotents) from NT2 (which has two); within $d=1$, commutativity is the sole separator;
      NT3 is commutative, while NT6 is not; all other invariants
      ($\Jac(A)^2$, $\LL(A)$, idempotent count, $d$) are equal for the two; within $d=2$, Loewy length distinguishes NT4
      ($\LL = 2$, $\Jac(A)^2 = 0$) from NT5 ($\LL = 3$, $\Jac(A)^2 \neq 0$).
\end{remark}

\section{Directions for Further Work}\label{sec:future}

The classification is specific to finite-dimensional associative unital
$\RR$-algebras of dimension three, and three features of this setting
are essential. Frobenius' theorem restricts the Wedderburn--Artin
summands over $\RR$ to matrix rings over $\RR$, $\CC$, or $\HH$; over
$\CC$ or a finite field the decomposition differs. The effect is already
visible in dimension three: $\RR\times\CC$ has no analogue over $\CC$
since $\CC$ does not split as a non-trivial product of $\CC$-algebras,
and accordingly \cite{Kobayashi2021} lists five unital algebras over
$\CC$ versus six over $\RR$. Finite-dimensionality underlies both the
Loewy-series bound $d\leq 2$ and Nakayama's Lemma; neither holds in
infinite dimensions. The Jacobson radical and Wedderburn--Artin theorem
presuppose associativity; Lie, Jordan, and octonion algebras require
different methods \cite{Schafer1966,McCrimmon2004,Baez2002}.

The next problem is the classification over finite fields.
Over $\FF_q$, Wedderburn's little theorem replaces Frobenius' theorem
in the semisimple stratum, and the cubic extension $\FF_{q^3}$ enters
as a semisimple algebra with no real counterpart.
The sign obstruction used in the $d=1$ stratum --- the equation
$c^2 = -1$ having no real solution --- is genuinely specific
to $\RR$: over $\FF_q$ with $q \equiv 1 \pmod{4}$ the equation is
solvable, and the obstruction, if there is one, must be found elsewhere.

\section*{Acknowledgements}

The author thanks the organizers of the AMS Special Session on Quaternions
at the January 2024 Joint Mathematics Meetings in San Francisco, at which
the trinion presentation \cite{Shelton2024} was given; their questions,
particularly regarding what occupies dimension three in the absence of a
division algebra, are what set this work in motion.

\end{document}